\documentclass[12pt, a4paper]{article}
\usepackage{amsfonts,amssymb,amsmath,amscd,latexsym,makeidx,theorem,color}
\usepackage{pdfsync}
\usepackage{color}
\usepackage{comment}

\newtheorem{thm}{Theorem}
\newtheorem{thmx}{Theorem}
\newtheorem{lem}[thm]{Lemma}
\newtheorem{prop}[thm]{Proposition}
\newtheorem{cor}[thm]{Corollary}
\newtheorem{defn}[thm]{Definition}

\newtheorem{OP}{Open problem}
\newenvironment{proof}{\noindent\emph{Proof.}}{\hfill$\square$\medskip}

\newcommand{\C}{\mathcal{C}}
\newcommand{\rn}{\mathbb{R}^{2m}}
\newcommand{\N}{\mathbb{N}}
\newcommand{\vol}{\mathrm{vol}}
\newcommand{\de}{\partial}
\newcommand{\ve}{\varepsilon}
\newcommand{\R}{\mathbb{R}}

\newcommand{\loc}{\mathrm{loc}}

\newcommand{\D}{\Delta}
\newcommand{\vp}{\varphi}

\title{Large blow-up sets for the prescribed $Q$-curvature equation in the Euclidean space}
\author{Ali Hyder\thanks{The authors are supported by the Swiss National Science Foundation, project nr. PP00P2-144669.} \\ {\small Universit\"at Basel}\\ {\footnotesize \texttt{ali.hyder@unibas.ch} }\and Stefano Iula${}^*$ \\ {\small Universit\"at Basel}\\ {\footnotesize \texttt{stefano.iula@unibas.ch} }\and Luca Martinazzi${}^*$ \\ {\small Universit\"at Basel}\\ {\footnotesize \texttt{luca.martinazzi@unibas.ch} }}

\begin{document}

\maketitle

\begin{abstract} 
Let $m\ge 2$ be an integer. For any open domain $\Omega\subset\rn$,  non-positive function $\varphi\in C^\infty(\Omega)$ such that $\Delta^m \varphi\equiv 0$, and bounded sequence $(V_k)\subset L^\infty(\Omega)$ we prove the existence of a sequence of functions $(u_k)\subset C^{2m-1}(\Omega)$ solving the Liouville equation of order $2m$
$$(-\Delta)^m u_k = V_ke^{2mu_k}\quad \text{in }\Omega, \quad \limsup_{k\to\infty} \int_\Omega e^{2mu_k}dx<\infty,$$
and blowing up exactly on the set $S_{\varphi}:=\{x\in \Omega:\varphi(x)=0\}$, i.e.
$$\lim_{k\to\infty} u_k(x)=+\infty \text{ for }x\in S_{\varphi} \text{ and }\lim_{k\to\infty} u_k(x)=-\infty \text{ for }x\in \Omega\setminus S_{\varphi},$$
thus showing that a result of Adimurthi, Robert and Struwe is sharp.
We extend this result to the boundary of $\Omega$ and to the case $\Omega=\rn$. Several related problems remain open.
%Similarly, when $\Omega$ is replaced by a closed Riemannian manifold $(M,g)$ with $Q$-curvature $Q_g$ and GJMS-operator $P_g$ such that $\ker P_g$ contains a non-constant function $\varphi$, given a bounded sequence $(Q_k)\subset L^\infty(M)$ we construct a sequence of  Riemannian metrics $(g_k)$ of the form $g_k=e^{2u_k}g$, where $u_k\in C^{2m-1}(M)$  solves the prescribed $Q$-curvature equation
%$$P_g u_k+Q_g= Q_ke^{2mu_k},$$
%such that $(g_k)$ blows up on $S_{\varphi}:=\{x\in M:\varphi(x)=\max_M \varphi\}$ and vanishes on $M\setminus S_{\varphi}$ as $k\to\infty$.
%This shows that the assumption $\ker P_g=\{constants\}$, used in several results on the prescribed $Q$-curvature equation, is necessary.
\end{abstract}

\section{Introduction and main results}

In several nonlinear elliptic problems of second order and ``critical type'', sequences of solutions are not always compact, as they can blow up at finitely many points, see e.g  \cite{AS}, \cite{BC}, \cite{BM}, \cite{D}, \cite{SU}, \cite{str84}, \cite{str85}. For instance, as shown by H. Br\'ezis and F. Merle in \cite{BM}:

\begin{thmx}[\cite{BM}]\label{thmBM}
Given a sequence $(u_k)_{k\in\N}$ of solutions to the Liouville equation
\begin{equation}\label{BMeq}
-\Delta u_k=V_ke^{2u_k}\quad \text{in }\Omega\subset\R^{2},
\end{equation}
with $\|V_k\|_{L^\infty}\le C$ and $\|e^{2u_k}\|_{L^1}\le C$ for some $C$ independent of $k$, there exists a finite (possibly empty) set $S_1=\left\{ x^{(1)},\dots, x^{(I)}\right\}\subset\Omega$ such that, up to extracting a subsequence one of the following alternatives holds:
\begin{itemize}
\item[$(i)$] $(u_k)$ is bounded in $C^{1,\alpha}_{\loc}(\Omega\setminus S_1)$.
\item[$(ii)$]  $u_k\to-\infty$ locally uniformly in $\Omega\setminus S_1$.
\end{itemize}
\end{thmx}

A similar behaviour is also found on manifolds, or in higher order and higher dimensional problems, see e.g. \cite{Mar3}, \cite{str07}, or even in $1$-dimensional situations involving the operator $(-\Delta)^\frac12$, see \cite{DLM}, \cite{DLMR}.
Now consider the problem
\begin{equation}\label{eq1}
(-\Delta)^m u_k=V_ke^{2mu_k}\quad  \text{in }\Omega\subset \R^{2m}
\end{equation}
\begin{equation}\label{vol}
\limsup_{k\to\infty} \int_{\Omega}e^{2mu_k}\, dx <\infty, \quad \limsup_{k\to\infty} \|V_k\|_{L^\infty(\Omega)}<\infty,
\end{equation}
We recall that \eqref{eq1} is a special case of the prescribed $Q$-curvature equation on a Riemannian manifold $(M,g)$ of dimension $2m$
\begin{equation}\label{eqpan}
P_g u+Q_g=Q_{g_u}e^{2mu}\quad \text{in }M,
\end{equation}
where $P_g=(-\Delta_g)^m+l.o.t.$ and $Q_g$ are the GJMS-operator of order $2m$ and the $Q$-curvature of the metric $g$, respectively, and $Q_{g_u}$ is the $Q$-curvature of the conformal metric $g_u:= e^{2u}g$. In this sense every solution $u_k$ to \eqref{eq1} gives rise to a metric $e^{2u_k}|dx|^2$ on $\Omega$ with $Q$-curvature $V_k$, and volume $\int_\Omega e^{2mu_k}dx$.

Since blow-up at finitely many points appears in many problems with various critical nonlinearities and also of higher order, one might suspect that this is a general feature also holding for \eqref{eq1}. On the other hand Adimurthi, Robert and Struwe \cite{ARS} found an example of solutions to \eqref{eq1}-\eqref{vol} for $m=2$ that blow up on a hyperplane, and showed in general that the blow-up set of a sequence $(u_k)$ of solutions to \eqref{eq1}-\eqref{vol} can be of Hausdorff dimension $3$. This was generalized to the case of arbitrary $m$ in \cite{Mar1}. More precisely for a finite set $S_1\subset\Omega\subset\R^{2m}$ let us introduce
\begin{equation}\label{defK}
\mathcal{K}(\Omega, S_1):=\{\varphi\in C^\infty(\Omega \setminus S_1):\varphi\le 0,\,\varphi\not\equiv 0,\, \Delta^m \varphi\equiv 0\},
\end{equation}
and for a function $\varphi \in \mathcal{K}(\Omega,S_1)$ set
\begin{equation}\label{defS0}
S_\varphi:=\{x\in\Omega\setminus S_1: \varphi(x)=0\}.
\end{equation}

\begin{thmx}[\cite{ARS,Mar1}]\label{trmARSM}
Let $(u_k)$ be a sequence of solutions to \eqref{eq1}-\eqref{vol} for some $m\ge 1$. Then the set
$$S_1:=\left\{x\in \Omega: \lim_{r\downarrow 0}\limsup_{k\to\infty}\int_{B_r(x)}|V_k|e^{2mu_k}dy\ge \frac{\Lambda_1}{2} \right\},\quad \Lambda_1:=(2m-1)!\vol(S^{2m})$$
is finite (possibly empty) and up to a subsequence either
\begin{itemize}
\item[$(i)$] $(u_k)$ is bounded in $C^{2m-1,\alpha}_{\loc}(\Omega\setminus S_1)$, or
\item[$(ii)$] there exists a function $\varphi\in\mathcal{K}(\Omega, S_1)$ and a sequence $\beta_k\to\infty$ as $k\to+\infty$ such that
$$\frac{u_k}{\beta_k}\to \varphi \text{ locally uniformly in }\Omega\setminus S_1.$$
In particular $u_k\to -\infty$ locally uniformly in $\Omega\setminus (S_\varphi\cup S_1)$.
\end{itemize}
\end{thmx}

Notice that Theorem \ref{trmARSM} contains Theorem \ref{thmBM} since  when $m=1$ we have $S_\varphi=\emptyset$ for every $\varphi \in \mathcal{K}(\Omega,S_1)$ by the maximum principle. In fact the more complex blow-up behaviour of \eqref{eq1} when $m>1$ can be seen as a consequence of the size of $\mathcal{K}(\Omega,S_1)$. A way of recovering a finite blow-up behaviour for \eqref{eq1}-\eqref{vol} was given by F. Robert \cite{Rob1} when $m=2$ and generalized by the third author \cite{Mar2} when $m\ge 3$, by additionally assuming 
$$\|\Delta u_k\|_{L^1(B_r(x))}\le C\quad \text{on some ball }B_r(x)\subset \Omega,$$
which is sufficient to control the ``polyharmonic part'' of $u_k$.

%Before stating our main results we recall that problem \eqref{eq1} can be seen as a special case of a geometric problem arising in conformal geometry as the higher-dimensional equivalent of the prescribed Gaussian curvature equation. Namely, let $(M,g)$ be a Riemannian manifold  of dimension $2m$ with GJMS-operator $P_g$ of order $2m$ and $Q$-curvature $Q_g$. If  $u\in C^{2m}(M)$ solves
%then the conformal metric $g_u=e^{2u}g$ has $Q$-curvature equal to $Q$. For more details we refer to the introduction in \cite{Mar1} and the references therein. Here we just remark that when $g=|dx|^2$ is the Euclidean metric on $\Omega\subset\rn$, $P_g=(-\Delta)^m$ and $Q_g=0$, so that \eqref{eqpan} indeed generalises \eqref{eq1}.

\medskip

The first result that we will prove shows that the condition given in \cite{ARS} and \cite{Mar1} on the set $S_\varphi$ above is sharp, at least when $S_1=\emptyset$. In fact we shall consider a slightly stronger result, by defining
\begin{equation}\label{defS0*}
S_\varphi^*:=S_\varphi\cup\{x\in \de\Omega: \lim_{\Omega\ni y\to x}\varphi(y)=0\},
\end{equation}
namely we add to $S_\varphi$ the points on $\de\Omega$ where $\varphi$ can be continuously extended to $0$. Then we have

\begin{thm}\label{thm-1}
Let $\Omega\subset \R^{2m}$, $m\ge 1$, be an open (connected) domain and let $(V_k)\subset L^\infty(\Omega)$ be bounded. Then for every  $\varphi\in \mathcal{K}(\Omega, \emptyset)$
there exists a sequence $(u_k)$ of solutions to \eqref{eq1} with 
\begin{equation}\label{vol0}
\int_{\Omega} e^{2m u_k}\, dx\to 0,
\end{equation}
 %of the form
%$$u_k= k\varphi +k(k)+v_k,$$
%where $k(k):\N \to \N$, $v_k\in C^{2m-1}(\bar\Omega)$, 
such that as $k\to \infty$
% $$\sup_{\Omega}|v_k|\to 0,\quad k(k)\to \infty,$$
% and
\begin{equation}\label{ukinfty}
u_k\to-\infty \text{ loc. unif. in $\Omega\setminus S_\varphi$},\quad u_k \to +\infty \text{ loc. unif. on $S_\varphi^*$,}
\end{equation}
where $S_\varphi$ and $S_\varphi^*$ are as in \eqref{defS0} and \eqref{defS0*}.
\end{thm}

The proof of Theorem \ref{thm-1} is based on a Schauder's fixed-point argument.
The case when $\Omega$ is smoothly bounded is very elementary, as one looks for solutions of the form
$$u_k= c_k\varphi +k+v_k,\quad c_k\to\infty,$$
where $v_k$ is a small correction term.

The general case is a priori more rigid. For instance in the case $m=1$, when $V_k\equiv 1$ there are few solutions to \eqref{eq1}-\eqref{vol} when $\Omega=\R^2$ (see \cite{CL}) and many more when $\Omega$ is bounded (see \cite{XC}).
To treat the general case we will borrow ideas from \cite{WY} (see also \cite{H-M}) and suitably prescribe the asymptotic behavior of $u_k$ at infinity. More precisely we will look for solutions of the form
$$u_k= c_k\varphi +k-\alpha_k \log(1+|x|^2)- \beta |x|^2+v_k,$$
for some $c_k\to\infty$, $\alpha_k\to 0$, $\beta>0$, and a function $v_k\to 0$ uniformly. If $\varphi(x)\to -\infty$ sufficiently fast as $|x|\to\infty$, or when $\Omega$ is bounded, one can choose $\beta=0$, but the case $\Omega=\R^{2m}$, $\varphi(x_1,\dots,x_{2m})=-x_1^2$ shows that $\beta$ in general must be positive when
$$\liminf_{ x\in \Omega, |x| \to \infty}\varphi(x)> -\infty,$$
otherwise the condition \eqref{vol} might fail to be satisfied.
 
%Fix $u_0\in C^\infty(\rn)$, $u_0>0$, such that $u_0(x)=\ln|x|$ for $|x|\ge 2$, and notice that integration by parts yelds
%$$ \int_{\rn}(-\Delta)^m u_0\, dx=-\gamma_m,$$
%where $\gamma_m$ is defined by
%\begin{equation}\label{gammam}
%(-\Delta)^m\log\frac{1}{|x|}=\gamma_m\delta_0 \text{ in }\rn, \text{ i.e. }\gamma_m=\frac{(2m-1)!}{2}\vol(S^{2m}).
%\end{equation}

%\begin{thm}\label{mainthmrn}
%Let $\Omega=\rn$, $(V_k)\subset L^\infty(\rn)$ be bounded, $\varphi\in C^\infty(\rn)$ satisfy
%\begin{equation}\label{condphi}
%\Delta^m\varphi\equiv 0,\quad \varphi\le 0, \quad \lim_{|x|\to+\infty}\varphi(x)=-\infty,
%\end{equation}
%and set $S_{\varphi}:=\{x\in \rn: \varphi(x)=0\}$.
%Consider the following problem
%\begin{equation}\label{eq1b}
%(-\Delta)^mu_k=Qe^{2mu_k}\quad \text{in }\rn,
%\end{equation}
%\begin{equation}\label{eq1bvol}
%\limsup_{k\to\infty} \int_{\rn}e^{2mu_k}\, dx <\infty.
%\end{equation}
%There exists a sequence $(u_k)_{k\in\N}$ of solutions to \eqref{eq1} of the form
%\begin{equation*}%\label{formukrn}
%u_k = k\varphi +k(k) - \alpha_k u_0 + v_k
%\end{equation*}
%for some
%$$\alpha_k>0,\quad v_k\in C_0(\rn):=\{v\in C^0(\rn):\lim_{x\to\infty} v(x)=0\},\quad k(k):\N\to \N$$
%with 
%\begin{equation}\label{formukrn2}
%k(k)\to+\infty,\quad\sup_{\rn}|v_k|\to 0,\quad \alpha_k\to 0,\quad \text{as } k\to+\infty,
%\end{equation}
%and such that \eqref{propofukbd2} holds as $k\to \infty$.
%\end{thm}
The simplicity of the proof of Theorem \ref{thm-1} comes at the cost of not being able to prescribe the total $Q$-curvature of the metric $g_{u_k}:=e^{2u_k}|dx|^2$, which will necessarily go to zero, together with the volume of $g_{u_k}$. % quantity
%$$\int_\Omega e^{2mu_k}dx,$$
%the volume of the conformal metric $e^{2u_k}|dx|^2$, which will necessarily go to zero.
Resting on variational methods from \cite{A-H} going back to \cite{CC},
we can extend Theorem \ref{thm-1} to the case in which we prescribe both the blow-up set $S_{\varphi}$ and the total curvature of the metrics $g_{u_k}$. This time, though, we will have to restrict to non-negative functions $V_k$.

\begin{thm}\label{thm-2}
Let $0<\Lambda<\Lambda_1/2$, $\Omega\subset \rn$ open, $m\ge 2$, $\varphi\in \mathcal{K}(\Omega,\emptyset)$, and let $S_\varphi$ be as 
in \eqref{defS0}. Let further $V_k:\Omega\to \R$ be functions for which there exists $x_0\in S^*_\varphi$ such that  
%$$0<a\le V_k\le b<\infty$$
\begin{equation}\label{condVk}
\liminf_{k\to+\infty}\int_{B_\ve(x_0)\cap \Omega} V_k \, dx>0,\quad\text{for every }\ve>0,\quad 0\leq V_k\leq b<\infty.
\end{equation}
%$$\Delta^m\varphi\equiv 0,\quad \varphi\le 0,\quad \varphi\not\equiv 0,$$
%and set $S_{\varphi}:=\{x\in \Omega: \varphi(x)=0\}$.
Then there exists a sequence $(u_k)_{k\in\N}$ of solutions to  \eqref{eq1} with 
\begin{align}\label{eq-2}
\int_{\Omega}V_k e^{2mu_k}dx=\Lambda,
\end{align}
such that \eqref{ukinfty} holds.
%$$u_k\to-\infty \text{ loc. unif. in $\Omega\setminus S_\varphi$},\quad u_k \to +\infty \text{ loc. unif. on }S_\varphi.$$
%$u_k(x)\to\infty$ for every $x\in S_{\varphi}$ and $u_k\to-\infty$ locally uniformly in in $\Omega\setminus S_{\varphi}$.
 \end{thm}

The integral assumption in \eqref{condVk} is crucial. In fact, for any $\vp\in \mathcal{K}(\Omega,\emptyset)$
there are functions $V_k$ satisfying $0\leq V_k\leq b<\infty$,
such that for every $\Lambda>0$ there exists no sequence $(u_k)$ of solution to \eqref{eq1} satisfying \eqref{ukinfty} and \eqref{eq-2} (see Proposition \ref{non}).

As we shall see, Theorems \ref{thm-1} and \ref{thm-2} give several examples of solutions blowing-up on the boundary, already in dimension $2$.

\begin{cor}\label{corollary} Let $\Omega\subset\R^{2m}$ with $m\geq 1$ be a bounded domain  with smooth boundary and let $\Gamma\subset \de\Omega$ be a proper closed subset.  Let $(V_k)$ be as in Theorem \ref{thm-1}. Then we can find solutions $u_k: \Omega\to\R$ to \eqref{eq1} such that the conclusion of Theorem \ref{thm-1} holds with $S^*_\varphi=\Gamma$ for some $\varphi\in \mathcal{K}(\Omega,\emptyset)$.
If $m\ge 2$ and $(V_k)$ additionally satisfies \eqref{condVk} for some $x_0\in\Gamma$, then we can prescribe \eqref{eq-2} instead of \eqref{vol0}.

 %$u_k\to \infty$ on $\Gamma$ and $u_k\to-\infty$ on $\bar \Omega \setminus \Gamma$.
\end{cor}

\begin{OP}
Can one remove the assumption $\Lambda<\frac{\Lambda_1}{2}$ in Theorem \ref{thm-2}? In the radially symmetric case this appears to be the case, as the following result shows.
\end{OP}

\begin{thm}\label{thm-4}
 Let $\Omega=B_{R_2}\setminus B_{R_1}\subset\R^{2m}$ and  $\varphi\in \mathcal{K}(\Omega,\emptyset)$ be radially symmetric.  Let $\Lambda>0$ and let $(V_k)$ be radially symmetric satisfying \eqref{condVk}. Then there exists a sequence of radially symmetric solutions $(u_k)$ to \eqref{eq1}  such that  \eqref{ukinfty} and \eqref{eq-2} hold. For $\Omega = B_R$ the same conclusion holds if in addition we have $\Delta \varphi(0)>0$ and $V_k\to1$ in $L^\infty(B_\delta(0))$ for some $\delta>0$.
 
%  \begin{align}\label{eq-3}
%  \int_{\Omega}V_ke^{2mu_k}dx=\Lambda,
%  \end{align}
% and
%$$u_k\to-\infty \text{ loc. unif. in $\Omega\setminus S_\varphi$},\quad u_k \to +\infty \text{ loc. unif. on }S_\varphi.$$
\end{thm}

\subsection*{Gluing open problems}

We have worked under the assumption $S_1=\emptyset$. What happens if we drop it?

\begin{OP}
Can one have both $S_1\neq \emptyset$ and $S_\varphi\neq \emptyset$ in Theorem B? Or when $m=1$ can one have $S_1\neq \emptyset$ and $S_\varphi^*\neq \emptyset$? 
\end{OP}

This can be considered as a gluing problem. For instance, can one glue a standard bubble of the form 
\begin{equation}\label{spherical}
u_{x_0,\lambda}(x):=\log \frac{2\lambda}{1+\lambda^2|x-x_0|^2},\quad \text{for some }\lambda>0,\, x_0\in \R^{2m}
\end{equation}
solving
\begin{equation}\label{liouentire}
(-\Delta)^m u_{x_0,\lambda}=(2m-1)!e^{2mu_{x_0,\lambda}},\quad (2m-1)!\int_{\R^{2m}}e^{2m u_{x_0,\lambda}}dx=\Lambda_1,
\end{equation}
to one of the solutions provided by Theorems \ref{thm-1} and \ref{thm-2}?

Moreover, as shown by Chang-Chen \cite{CC}, when $m\ge 2$ problem \eqref{liouentire} has several solutions which are not of the form \eqref{spherical}. Such solutions behave polynomially at infinity, as shown in \cite{Lin,Mar0} (see also \cite{hyd,JMMX} for similar results in odd dimension). Let us call $v$ such a solution and
$$v_{x_1,\mu}(x):=v(\mu (x-x_1))+\log\mu, \quad \text{for some }\mu>0,\, x_1\in \R^{2m}.$$

\begin{OP} Can one glue a spherical solution $u_{x_0,\lambda}$ to a non-spherical solution $v_{x_1,\mu}$ as above ($x_1\ne x_0$)? More precisely, can one find a sequence of solutions $(u_k)$ to \eqref{eq1}-\eqref{vol} with $u_k=u_{x_0,\lambda_k}+w_k$ suitably close to $x_0$ and $u_k=v_{x_1,\mu_k}+w_k$ suitably close to $x_1$, with an error term $w_k$ bounded and $\lambda_k,\mu_k\to\infty$?
\end{OP}

This problem can be seen in terms of gradient estimates or estimates for $\Delta u_k$. Indeed on any fixed ball $B$ one has
$$\|\Delta u_{\lambda,x_0}\|_{L^1(B)}=O(1),\quad \|\Delta v_{\lambda,x_1}\|_{L^1(B)}\to \infty,\quad \text{as }\lambda\to\infty$$
(see Theorems 1 and 2 in \cite{Mar0}). This is consistent with a result of F. Robert \cite{Rob1}, extended in \cite{Mar2}, stating that in a region $\Omega_0$ such that $\|\Delta u_k\|_{L^1(\Omega_0)}\le C$, $u_k$ has a bubbling behaviour leading to solutions of the form \eqref{spherical}.

It was open whether there exists a sequence $(u_k)$ of solutions to \eqref{eq1}-\eqref{vol} on some domain $\Omega$ in $\R^{2m}$ with 2 open regions $\Omega_0,\Omega_1\subset\Omega$ such that
$$\|\Delta u_k\|_{L^1(\Omega_0)}=O(1),\quad \|\Delta u_k\|_{L^1(\Omega_1)}\to \infty.$$
We will prove that this is actually possible.
\begin{thm}\label{thm-5}
On $\Omega=B_2\subset\R^{2m}$ for any $\Lambda\in (0,\Lambda_1)$ we can find a sequence $(u_k)$ of solutions to \eqref{eq1}-\eqref{vol} with $V_k\equiv 1$ such that
\begin{equation}\label{Qcurv2}
\int_{B_2}e^{2mu_k}dx=\Lambda,
\end{equation}
and
\begin{equation}\label{Deltauk}
\int_{B_1}|\D u_k|dx\leq C,\quad \int_{B_2}(\D u_k)^-dx\xrightarrow{k\to\infty}\infty.
\end{equation}
\end{thm}

It remains open whether in the situation of Theorem \ref{thm-5} one can also have blow-up in $B_1$, in $B_2\setminus B_1$, or in both regions.

\medskip

In what follows we will denote by $C$ a generic positive constant that can change its value from line to line.

\paragraph{Acknowledgements} The question that led to Theorem \ref{thm-1}, then extended into the present work, was raised by Michael Struwe to the third author several years ago.

\section{Proof of Theorem \ref{thm-1}}

In order to clarify the simple idea behind the proof we start considering the easier case when $\Omega$ is bounded and has regular boundary. The proof in the general case is more complex.

\subsection{Case $\Omega$ smoothly bounded}

The proof will be based on an application of a fixed-point argument. Consider the Banach space $$X:=C^0(\bar{\Omega}),\quad \|v\|_X=\max_{x\in\bar{\Omega}}|v(x)|.$$
%A function $u_k$ of the form $u_k=v_k+k\varphi+c_k$, is a solution of \eqref{eq1} if and only if $v_k$ solves
 %$$   (-\D)^m v_k=V_ke^{2mc_k}e^{2mk\vp}e^{2mv_k}\quad\text{in }\Omega.$$
 For each $k\in\N$ choose $c_k\geq k^2$ such that $$\|e^{2mc_k \vp}\|_{L^2(\Omega)}\leq e^{-3mk}.$$
For $k\in \N$ consider the operator $T_{k}:X\to X$ defined by $T(v)=\bar{v}$ where $\bar{v}$ is the unique solution of
 $$\left\{\begin{array}{ll}
           (-\D)^m\bar{v}=V_ke^{2m(k+c_k\vp+v)}&\quad\text{in }\Omega \\
           \bar{v}=\D\bar{v}=\dots=\D^{m-1}\bar{v}=0&\quad\text{on }\partial\Omega.
          \end{array}
\right.$$
 %For $k, k\in \N$ consider the operator $T_{k,k}:X\to X$ defined by $T_k(v)=\bar{v}$ where $\bar{v}$ is the unique solution of
 %$$\left\{\begin{array}{ll}
  %%         (-\D)^m\bar{v}=V_ke^{2mk}e^{2mk\vp}e^{2mv}&\quad\text{in }\Omega \\
  %         \bar{v}=\D\bar{v}=\dots=\D^{m-1}\bar{v}=0&\quad\text{on }\partial\Omega.
   %       \end{array}
%\right.$$
From elliptic estimates, the Sobolev embedding and Ascoli-Arzel\`a's theorem it follows that $T_{k}$ is compact.  
Moreover, for every $v\in X$ we have %Then by elliptic estimates and together with Sobolev embeddings we have 
$$\|\bar{v}\|_X\leq C_1 \|\D^m\bar{v}\|_{L^2(\Omega)}\leq C_2Me^{2mk}\|e^{2mv}\|_X\|e^{2mc_k\vp}\|_{L^2(\Omega)},\quad\|V_k\|_{L^\infty}\leq M.$$
This shows that 
\begin{align}\label{est-Tkl}
 \|T_{k}(v)\|_X\leq C_3e^{2mk}e^{-3mk},\quad\text{for }\|v\|_X\leq 1,\quad C_3:=C_2M.
\end{align}
%Since, $\vp\leq 0$ in $\Omega$ and $S_{\varphi}$ has measure zero, by the dominated convergence theorem $$\lim_{k\to\infty}\|e^{2mk\vp}\|_{L^2(\Omega)}=0.$$
%For each $k\in\N$ we choose $k_k\geq k^2$ such that
%$$\|e^{2mk_k \vp}\|_{L^2(\Omega)}\leq \frac{1}{2C_3}e^{-3mk}.$$ 
Therefore  $T_{k}(\bar{B}_1)\subset \bar{B}_\frac 12$ for $k$ large enough (here $B_r$ is a ball in $X$), and hence $T_{k}$ has a 
fixed point in $X$. We denote it by  $v_k$.  Notice that $\|v_k\|_X\leq C e^{-mk}\to 0$ as $k\to\infty$.
Moreover, by H\"older's inequality,
$$\int_{\Omega}e^{2mk}e^{2mc_k\vp}e^{2mv_k}dx\leq e^{2mk}\sqrt{|\Omega|}\|e^{2mc_k \vp}\|_{L^2(\Omega)}\xrightarrow{k\to\infty}0.$$
We set $$u_k:=v_k+k+c_k\vp.$$ Then $u_k$ satisfies 
$$(-\D)^mu_k=V_{k}e^{2mu_k}\quad\text{in }\Omega,\quad\int_{\Omega}e^{2mu_k}dx\xrightarrow{k\to\infty}0.$$
Moreover 
$$\inf_{x\in S_{\varphi}^*}u_k =o(1)+k \xrightarrow{k\to\infty}\infty.$$ 
Finally, for any compact subset $K\Subset \Omega\setminus S_{\varphi}$, using that $c_k\geq k^2$, we obtain
$$\max_{x\in K}u_k=o(1)+k+c_k \max_{x\in K}\vp \leq k -\ve k^2 \xrightarrow{k\to\infty}-\infty, $$ where 
 $\ve>0$ is such that  $\max_{x\in K}\vp<-\ve$. This completes the proof.

%%%%%%%%%%%%%%%%%%%%%%%%%%%%%%%%%%%%%%%%%%%%
%THE CASE OF R2m

\subsection{General case}

%In this section we consider $\Omega=\rn$. % In this setting Problem \eqref{eq1}-\eqref{vol} then read as \eqref{eq1b}-\eqref{eq1bvol}.
We will use many ideas from \cite{H-M} and \cite{WY}. Let $\varphi\in \mathcal{K}(\Omega,\emptyset)$. 
%A Liouville-type theorem (see for instance \cite[Theorem 5]{Mar0}) implies the $\varphi$ is a polynomial, and a result of Gorin \cite[Theorem $3.1$]{gor}, implies for $|x|$ large we have
%\begin{equation}\label{polinfty2}
%\varphi(x)\le -|x|^\eta  
%\end{equation}
%for some positive $\eta >0$.
Fix $u_0\in C^\infty(\rn)$, $u_0>0$, such that $u_0(x)=\log|x|$ for $|x|\ge 2$, and notice that integration by parts yields
\begin{equation}\label{unot}
 \int_{\rn}(-\Delta)^m u_0\, dx=-\gamma_{2m},
\end{equation}
where $\gamma_{2m}$ is defined by
\begin{equation}\label{gammam}
(-\Delta)^m\log\frac{1}{|x|}=\gamma_{2m}\delta_0 \text{ in }\rn, \text{ i.e. } \gamma_{2m}=\frac{\Lambda_1}{2}.
\end{equation}

We will work in weighted spaces.

\begin{defn}\label{definitionspaces} For $k\in \mathbb{N}$, $\delta\in\mathbb{R}$ and $p\ge 1$ we set $M_{k,\delta}^p(\rn)$ to be the completion of $C_c^{\infty}(\rn)$ in the norm 
$$\|f\|_{M_{k,\delta}^p} := \sum_{|\beta|\leq k}\|(1+|x|^2)^{\frac{(\delta+|\beta|)}{2}}D^{\beta}f\|_{L^p(\rn)}.$$
We also set $L^p_{\delta}(\rn):= M_{0,\delta}^p(\rn)$. Finally we set 
$$\Gamma^p_\delta(\rn)  := \left\{f\in L^p_{2m+\delta}(\rn):\int_{\rn}f dx=0\right\},$$
whenever $\delta p>-2m$, so that $L^p_{2m+\delta}(\rn)\subset L^1(\rn)$ and the above integral is well defined.
\end{defn}

\begin{lem}[Theorem 5 in \cite{McOwen}]\label{isomorphism} 
For $1<p<\infty$ and $\delta\in \left(-\frac{2m}{p},-\frac{2m}{p}+1\right)$, the operator $(-\Delta)^m$ is an isomorphism from $M_{2m,\delta}^p(\rn)$ to $\Gamma^p_\delta(\rn)$. 
\end{lem}

\begin{lem}[Lemma $2.3$ in \cite{H-M}]\label{compact}
For $\delta>-\frac{2m}{p}$, $p\ge 1$, the embedding  
$$E:M_{2m,\delta}^p(\rn)\hookrightarrow C_0(\rn)$$
is compact.
\end{lem}

We will construct a sequence $(u_k)_{k\in\N}$ of solutions to \eqref{eq1}-\eqref{vol0} of the form
\begin{align}
 u_k= -\beta|x|^2+c_k\varphi-\alpha_k u_0 +k+v_k, \quad \text{in }\Omega,\label{exp-uk}
\end{align}
for some $\beta\geq 0$ and
$v_k\in C^{2m-1}(\R^{2m})$ such that as $k\to \infty$
 $$\sup_{\Omega}|v_k|\to 0,\quad c_k\to \infty,\quad \alpha_k\to 0.$$
In general $\beta>0$ is an arbitrary fixed constant, but if $\vp$ satisfies \begin{equation}\label{condphi2}
\int_\Omega e^{2m\varphi}|x|^{2q}\, dx<\infty,\quad \text{for some }q>0,
\end{equation} 
then we can take $\beta=0$ as well.

We consider $$X:=C_0(\R^{2m}):=\left\{v\in C^0(\R^{2m}):\lim_{|x|\to\infty}v(x)=0\right\},\quad \|v\|_X=\sup_{x\in{\R^{2m}}}|v(x)|.$$ 
For $c\in\R$ we set
\[
F_{k,c}=\left\{
\begin{array}{ll}
V_ke^{2mk}e^{-2m\beta|x|^2}e^{2mc\vp}&\text{in }\Omega\\
0&\text{in }\R^{2m}\setminus \Omega.
\end{array}
\right.
\]
Let $\ve_1\in (0,\frac{q}{8m})$ (to be fixed later). We fix $p>1$ and $\delta\in(-\frac{2m}{p},\frac{2m}{p}+1)$ such that $p(2m+\delta)<\frac q4$.  For each $k\in \N$ we choose $c_k\geq k^2$ so that  
\begin{equation}\label{noname1} 
\int_{\R^{2m}}|F_{k,c_k}(x)|(M+|x|)^{q}dx\leq \ve_1e^{-k}e^{-2m},%\quad e^{u_0}\leq M\,\text{on }B_2,
\end{equation}
\begin{equation}\label{noname2}
%\|F_k(M+|x|)^{\frac q4}\|_{L^1(\R^{2m})}+
\|F_k(M+|x|)^{\frac q4}\|_{L^p_{2m+\delta}}\leq \ve_1e^{-k},\quad F_k:=F_{k,c_k},
\end{equation}
\begin{equation}\label{noname3}
\int_\Omega e^{2m(c_k\varphi +k)}(M+|x|)^q\, dx\leq e^{-k},
\end{equation}
where $q$ is as in \eqref{condphi2} and $M>0$ is such that $e^{u_0}\leq M\,\text{on }B_2$. 
For each $k\in\N$, define a continuous function  $I_k$ on $X\times(-\frac{q}{2m},\frac{q}{2m})$ given by
$$I_{k}(v,\alpha)=\frac{1}{\gamma_{2m}}\int_{\R^{2m}}F_{k}e^{-2m\alpha u_0}e^{2mv}dx.$$
If $I_{k}(v,0)>0$ then 
$$\lim_{\alpha\to 0^+}\frac {I_{k}(v,\alpha)}{\alpha} =\infty,\quad \frac {I_{k}(v,\ve_1e^{-k})}{\ve_1e^{-k}}\leq 1,\quad \|v\|_X\leq 1,$$ and hence there exists
$\alpha\in (0,\ve_1e^{-k}]$ such that $I_{k}(v,\alpha)=\alpha$.  Notice that 
$$\sup_{\alpha\in [-\frac{q}{4m},0]}|I_k(v,\alpha)|\leq e^{-k}\varepsilon_1,\quad\text{for }\|v\|_X\leq 1.$$ Thus, if $I_{k}(v,0)<0$ then 
$$\lim_{\alpha\to 0^-}\frac {I_{k}(v,\alpha)}{\alpha} =\infty,\quad \frac {|I_{k}(v,-\ve_1e^{-k})|}{\ve_1e^{-k}}\leq 1,\quad \|v\|_X\leq 1,$$
and hence there exists $\alpha\in [-\ve_1e^{-k},0)$ such that $I_{k}(v,\alpha)=\alpha$. For $\|v\|_X\leq 1$ we define
$$
\alpha_{k,v}:=\left\{\begin{array}{ll}
                   \inf \{\alpha>0:\alpha=I_k(v,\alpha)\}&\quad\text{if }I_k(v,0)>0\\
                   \sup \{\alpha<0:\alpha=I_k(v,\alpha)\}&\quad\text{if }I_k(v,0)<0\\
                   0 &\quad\text{if }I_k(v,0)=0.
                     \end{array}
                     \right.
$$
%{\color{blue} Then  $v\mapsto \alpha_{k,v}$ is continuous in $B_1\subset X$.}
From the continuity of $I_k$ it follows that $\alpha_{k,v}=I_k(v,\alpha_{k,v})$.

\begin{lem}
There exists $\ve_0>0$ such that for every   $\ve\in(0,\ve_0)$ and  for every $v\in B_1$ if 
$$ I_k(v,\alpha_v)=\alpha_v\quad\text{ for some $|\alpha_v|<\frac{q}{4m}$},$$
then for every $w\in B_{\ve^2}(v)\cap B_1$ there exists $\alpha_w\in(\alpha_v-\ve,\alpha_v+\ve)$ such that 
$$ I_k(w,\alpha_w)=\alpha_w.$$
Moreover, the map $v\mapsto \alpha_{k,v}$ is continuous on $B_1$. 
\end{lem}
\begin{proof}
Let $R>0$ be such that $R^q=\frac{1}{\ve^2}$. With this particular choice of $R$ we have 
$$ \int_{B_R^c} |F_k| \left(1+|x|\right)^{q}\, dx\leq C\ve^2.$$
Now for $|\alpha_v-\alpha|(2m\log R)^2<\frac 12$ we have
\begin{equation*}
\begin{split} 
&\frac{1}{\gamma_{2m}}\int_{B_R}F_{k}e^{-2m\alpha u_0}e^{2mw}\, dx\\
&=\frac{1}{\gamma_{2m}}\int_{B_R}F_k e^{-2m\alpha_v u_0}e^{2mv}e^{2m(w-v)}e^{2m(\alpha_v-\alpha)u_0}\, dx\\
&=\frac{1}{\gamma_{2m}}\int_{B_R}F_k e^{-2m\alpha_v u_0}e^{2mv}\left(1+ 2m(\alpha_v-\alpha)u_0+O\left(\alpha_v-\alpha\right)\right)\left( 1+O(\ve^2)\right)\, dx\\
&=I_k(v,\alpha_v)+\frac{2m(\alpha_v-\alpha)}{\gamma_{2m}}(1+O(\ve^2))\int_{B_R}F_k e^{-2m\alpha_v u_0}e^{2mv}u_0\, dx\\
&\quad +O\left(\alpha_v-\alpha\right)\int_{B_R}F_k e^{-2m\alpha_v u_0}e^{2mv}\, dx+O(\ve^2)\\
&=:I_k(v,\alpha_v)+\frac{2m(\alpha_v-\alpha)}{\gamma_{2m}}(1+O(\ve^2))J_1+O\left(\alpha_v-\alpha\right)J_2+O(\ve^2).
\end{split}
\end{equation*}
Using \eqref{noname1} we get
\[
\begin{split}
|J_1|&\leq e^{2m}\int_{B_R}|F_k|e^{-2m\alpha_v u_0}u_0\, dx\leq e^{2m}\int_{B_R}|F_k| (M+|x|)^{\frac q2}u_0\, dx\\
&\leq C(q)e^{2m}\int_{B_R} |F_k|(M+|x|)^q\, dx\leq C(q) \ve_1,
\end{split}
\]
and $J_2=O(\ve_1)$. Let $\alpha=\alpha_v+\rho$, with $|\rho|\leq \frac{1}{2(2m\log R)^2}$. Then 
\[
I_k(w,\alpha_v+\rho)-(\alpha_v+\rho)=\rho+O(\ve^2)+\rho O(\ve_1).
\]
We fix $\ve_0>0$ and $\ve_1>0$ such that for every $\ve\in (0,\ve_0)$ we have  $|O(\ve^2)|\leq \frac{\ve}{4}$  and
$|O(\ve_1)|\leq \frac 14$.  Then we can choose $\bar\rho\in(-\ve,\ve)$ such that
$$|\bar\rho|\leq \frac{1}{2(2m\log R)^2},\quad \bar\rho+O(\ve^2)+\bar\rho O(\ve_1)=0,$$
concluding the first part of the lemma.

%Now we assume by contradiction that  is not continuous. 
Now we prove the continuity of the map  $v\mapsto \alpha_{k,v}$ from $B_1$ to $\R$.  

  For $v_n\to v\in B_1$ it follows that  (at least) for large $n$, 
$|\alpha_{k,v_n}|<\frac{q}{4m}$ and $|\alpha_{k,v}|<\frac{q}{4m}$. 
First we consider the case $\alpha_{k,v}=0$. Then for any $\ve>0$  one has 
$I_{k}(v_n,\alpha_{v_n})=\alpha_{v_n}$ for some $\alpha_{v_n}\in  (-\ve, \ve)$  where  $\|v-v_n\|_X<\ve^2$. This follows from 
the  first part of the lemma. Since $|\alpha_{k,v_n}|\leq |\alpha_{v_n}|$, we have the continuity. 

Now we consider $\alpha_{k,v}>0$ (negative case is similar). Then $I_k(v,0)>0$, and hence $\alpha_{k,v_n}\geq 0$ for large $n$. We set 
$\alpha_\infty:=\lim_{n\to\infty}\alpha_{k,v_n}$ (this limit exists at least for a subsequence). From the continuity of the map
$I_k$ it follows that $I_k(v,\alpha_\infty)=\alpha_\infty$. Since $\alpha_\infty\geq0$  and $I_k(v,0)>0$, we must have $\alpha_\infty>0$. From the definition
of $\alpha_{k,v}$ we deduce that $\alpha_{k,v}\leq \alpha_\infty$. 
We fix $\ve\in (0,\frac{\alpha_{k,v}}{2})$. Then by the first part of the lemma there exists 
$\alpha_{v_n}\in (\alpha_{k,v}-\ve,\alpha_{k,v}+\ve)$ such that $I_k(v,\alpha_{v_n})=\alpha_{v_n}$ for every
$\|v-v_n\|_X<\ve^2$. Since $\alpha_{k,v_n}\leq\alpha_{v_n}$ and $\alpha_{k,v_n}\to\alpha_\infty$, we have for $n$ large 
$$\alpha_{k,v}\leq \alpha_\infty\leq \alpha_{k,v_n}+\ve\leq \alpha_{v_n}+\ve\leq \alpha_{k,v}+2\ve.$$ We conclude the lemma. 
%If $\alpha_\infty\neq \alpha_{k,v}$ then we set 
%$\ve=\min\{\frac{\alpha_\infty}{4},\frac{\alpha_{k,v}}{4},\frac{|\alpha_{k,v}-\alpha_\infty|}{4},\frac{\ve_0}{2}\}$.  % then  we assume by contradiction that $\alpha_{k,v_n}\nrightarrow \alpha_{k,v}$. Then 
%for some $0<\ve<\frac{1}{10}|\alpha_{k,v}|$, we have (for a subsequence) $\alpha_{k,v_n}\not\in (\alpha_{k,v}-\ve, \alpha_{k,v}+\ve)$ for $n$ large. In fact,
%for large  $n$, we muast have $0<\alpha_{k,v_n}<\alpha_{k,v}-\ve$.  By the first
%part of the lemma, there exixts $\alpha_v\in (0,\alpha_{k,v_n}+\frac\ve2)$ such that $I_k(v,\alpha_v)=\alpha_v$ for  $n$ large. 
%Then $\alpha_{v_n}>0$ and   $\alpha_n>0$ for $n$ large. This gives a contradiction  as $\alpha_{k,v}\leq\alpha_n$ (when $\alpha_n>0$) and .
\end{proof}

\medskip 

\noindent
\emph{Proof of Theorem \ref{thm-1}}
We define $T_k:B_1\subset X\to X$, $v\mapsto \bar{v}$, where 
$$\bar v(x):=\frac{1}{\gamma_{2m}}\int_{\rn} \log\left(\frac{1}{|x-y|}\right) F_k(y) e^{-2m\alpha_{k,v}u_0+2mv(y)}\, dy+\alpha_{k,v}u_0,$$
that is $\bar v$ solves 
$$(-\D)^m\bar{v}=F_k e^{-2m\alpha_{k,v}u_0+2mv}+\alpha_{k,v}(-\D)^mu_0.$$ 
Notice that arguing as in \cite{H-M} one gets $\bar v\in X$. Using \eqref{unot} and our choice of $\alpha_{k, v}$ we have 
$$ \int_{\rn}(-\Delta)^m \bar v\, dx=0.$$
With our choice of $\delta$ and $p$ we have $\bar v\in M^p_{2m,\delta}(\rn)$.
For $v\in \bar{B}_1\subset X$ we bound with Lemma \ref{isomorphism}, Lemma \ref{compact} and \eqref{noname2}
\begin{align*}
 \|T_k(v)\|_X&\leq C_1\|T_k(v)\|_{M^p_{2m,\delta}}\leq C_1\|(-\D)^m\bar{v}\|_{\Gamma^p_\delta},\\
 &\leq C_1\|e^{-2m\alpha_{k,v}u_0}F_k\|_{L^p_{2m+\delta}}+C_1|\alpha_{k,v}|\|(-\D)^mu_0\|_{L^p_{2m+\delta}}\xrightarrow{k\to\infty}0.
\end{align*}
Therefore, for $\ve_1$ small enough, $\|T_k(v)\|_X\leq \frac 12$ and there exists a fixed point $v_k$ for every $k$. 
Hence,  thanks to \eqref{noname3}, the sequence %(depending on $F_k$)
$$u_k(x)=-\beta|x|^2-\alpha_{k,v_k}u_0(x)+c_k\vp(x)+k+v_k(x),\quad x\in \Omega, $$ is
a sequence of solutions with the stated properties.
\hfill $\square$

\section{Proof of Theorem \ref{thm-2} and Corollary \ref{corollary}}

A slightly different version of the following proposition appears in \cite{A-H}. For the sake of completeness we give a sketch of the proof. 

\begin{prop}\label{sol-prop}
Let $w_0(x)=\log\frac{2}{1+|x|^2}$ and consider two functions $K,f:\R^{2m}\to\R$ such that
$$K\geq 0,\quad K\not\equiv0,\quad Ke^{-2mw_0}\in L^\infty(\R^{2m})$$
and
$$f e^{-2mw_0}\in L^\infty(\R^{2m}),\quad \Lambda:=\int_{\R^{2m}}fdx\in (0,\Lambda_1).$$
Then there exists a function  $w\in C^{2m-1}(\R^{2m})$ and a constant $c_w$ such that 
 \begin{equation}\label{w-ell}
   (-\Delta)^mw=Ke^{2m(w+c_w)}- f \quad  \text{in $\R^n$}, \quad \int_{\R^{2m}}Ke^{2m(w+c_w)}dx=\Lambda,%\int_{\rn}Ke^{nw}<\infty,
  \end{equation}
  and $\lim_{|x|\to\infty} w(x)\in\mathbb{R}$. Moreover, if $f$ is of the form $f=(-\D)^mg $ for some $g\in C^{2m}(\R^{2m})$ with $g(x)=O(\log |x|)$ at infinity, then 
  $w$ satisfies 
  $$w(x)=\frac{1}{\gamma_{2m}}\int_{\R^{2m}}\log\left(\frac{1+|y|}{|x-y|}\right)K(y)e^{2m(w(y)+c_w)}-g(x)+C,$$ for some $C\in\R$.
\end{prop}
\begin{proof}
 Let $\pi$ be the stereographic  projection from $S^{2m}$ to $\R^{2m}$.
We define the functional $J$ on $H^m(S^{2m})$ given by
$$J(u)=\int_{S^{2m}}\left(\frac12|(P^{2m}u)^\frac12|^2+\tilde f_1u\right)dV_0-
       \frac{\Lambda}{2m}\log\left(\int_{S^{2m}}\tilde Ke^{-2mw_0\circ\pi}e^{2mu}dV_0\right),$$
       where $f_1:=fe^{-2mw_0}$, $\tilde f_1:=f_1\circ\pi$, $\tilde K:=K\circ\pi$ and $P^{2m}$ is the Paneitz operator of order $2m$ with respect to the standard metric on $S^{2m}$.
 Following the arguments in \cite{A-H} one can show that there exists $u\in H^{2m}(S^{2m})$ such that 
 $$P^{2m}u=\frac{\Lambda \tilde Ke^{-2mw_0\circ\pi}e^{2mu}}{\int_{S^{2m}}\tilde Ke^{-2mw_0\circ\pi}e^{2mu}dV_0}-\tilde f_1=:C_0\tilde Ke^{-2mw_0\circ\pi}e^{2mu}-\tilde f_1.$$
 Notice that $P^{2m}u\in L^\infty(S^{2m})$, thanks to the embedding  $H^{2m}(S^{2m})\hookrightarrow C^{0}(S^{2m})$, and hence $u\in C^{2m-1}(S^{2m})$.
 
   We set $w=u\circ\pi^{-1}$. Then  $w\in C^{2m-1}(\R^{2m})$ 
 and $\lim_{|x|\to\infty} w(x)\in\mathbb{R}$. %, thanks to the embedding $H^{2m}(S^{2m})\hookrightarrow C^{2m-1}(S^{2m})$.  
 Using the following identity of Branson (see \cite{Branson}) 
 $$(-\D)^m v=e^{2mw_0}(P^{2m}v)\circ\pi^{-1},\quad \text{for every }v\in C^\infty(S^{2m}),$$ and by an approximation argument, we have that 
 $$(-\D)^mw=C_0Ke^{2mw}-f=:Ke^{2m(w+c_w)}-f,\quad \text{in }\R^{2m}.$$ 
 
 Now we set 
 $$\tilde w(x):=\frac{1}{\gamma_{2m}}\int_{\R^{2m}}\log\left(\frac{1+|y|}{|x-y|}\right)K(y)e^{2m(w(y)+c_w)}-g(x).$$ 
 Then $\D^m(w-\tilde w)=0$ in $\R^{2m}$ and $(w-\tilde w)(x)=O(\log |x|)$ at infinity. Therefore, $w=\tilde w+C$ for some $C\in \R$. 
 
This finishes the proof of the proposition. 
 \end{proof}

\medskip
\noindent
\emph{Proof of  Theorem \ref{thm-2}}
Let $\vp\in \mathcal{K}(\Omega,\emptyset)$ and let  $u_0\in C^\infty(\R^{2m})$ be such that $u_0=-\log|x|$ on $B_1^c$. 
We set $f=\frac{2\Lambda}{\Lambda_1}(-\D)^m u_0$. For each $k\in\N$ we set 
$$K=K_k:=V_k e^{2m(-\beta|x|^2+k\vp+\alpha u_0)},\quad \alpha:=\frac{2\Lambda}{\Lambda_1},\quad \beta>0,$$ 
and we extend $K_k$ by $0$ outside $\Omega$. % Notice that $\alpha\in (0,2)$ if $\Lambda\in (0,\Lambda_1)$. Then $K$ and $f$ satisfy the hypothesis of Theorem \ref{thm-C}. Therefore, there exists a solution $w_k$ to \eqref{w-ell}, that is, 
%$$(-\Delta)^mw_k=K_k e^{2m(w_k+c_{w_k})}-\frac{2\Lambda}{\Lambda_1}(-\D)^m u_0\quad  \text{in $\R^{2m}$},
%\quad w_k(\infty):=\lim_{|x|\to\infty}w_k(x)\in \R,$$ 
%and 
%$$\int_{\R^{2m}}K_k e^{2m(w_k+c_{w_k})}dx=\Lambda.$$
%One can show that (see \cite[Lemma 3.1]{H-M})
Then by Proposition \ref{sol-prop} there exists a sequence of functions $(w_k)$ satisfying
$$w_k(x)=\frac{1}{\gamma_{2m}}\int_{\R^{2m}}\log\left(\frac{1+|y|}{|x-y|}\right)K_k(y) e^{2m(w_k(y)+c_{w_k})}dy-\frac{2\Lambda}{\Lambda_1}u_0+a_k,$$
for some $a_k\in \R$. 
We set $$u_k(x):=w_k +c_{w_k}-\beta|x|^2+k\vp(x)+\frac{2\Lambda}{\Lambda_1} u_0(x),\quad x\in\Omega\cup S_\vp^*.$$ 
Then $u_k$ satisfies %(below $V_k\equiv 0$ on $\Omega^c$)
\[\begin{split}
u_k(x)&=\frac{1}{\gamma_{2m}}\int_{\Omega}\log\left(\frac{1+|y|}{|x-y|}\right)V_k e^{2mu_k(y)}dy-\beta|x|^2+k\vp(x)+c_k\\
%&=:\frac{1}{\gamma_{2m}}I_k(x)-\beta|x|^2+k\vp(x)+c_k,
\end{split}
\]
and also \eqref{eq-2}, where $c_k:=a_k+c_{w_k}$. We conclude the proof with Lemma \ref{lemmause}.
%
%Arguing as in the previous section we prove that $c_k\to+\infty$ as $k\to+\infty$. 
%To prove that $u_k\to-\infty$ locally uniformly in $\Omega\setminus S_{\varphi}$ one can argue again as in the previous
%section.% taking $\Omega$ to be the ball $B_R$ for any $R>0$ fixed.
\hfill $\square$

%%%%%%%%%%%%%%%%%%%%%%%%%%%%%%%%%%%%%%%%%%%%%%%%%%%%%%%%%%%%%%%%%%%%%%%%%%%%%%%%%%%%%%%%%%%%%%%%%%%%%%%%%%%%%%%%%%
%%%%%%%%%%%%%%%%%%%%%%%%%%%%%%%%%%%%%%%%%%%%%%%%%%%%%%%%%

\begin{lem}\label{lemmause}
 Let $\Omega$ be a domain in $\R^{2m}$. Let $\varphi$ and $V_k$ as in Theorem \ref{thm-2}. Let $(u_k)$ be a sequence of solutions to 
 $$u_k(x)=\frac{1}{\gamma_{2m}}\int_{\Omega}\log\left(\frac{1+|y|}{|x-y|}\right)V_k e^{2mu_k(y)}dy-\beta|x|^2+k\vp(x)+c_k,\quad x\in\Omega\cup S_\vp^*,$$
% $$0<\bar\Lambda\leq \int_{\R^{2m}}V_k e^{2mu_k(y)}dy\leq \Lambda<\frac{\Lambda_1}{2},$$ 
for some $\beta> 0$. %$\vp\in \mathcal{K}(\Omega,\emptyset)$, 
% $0\leq V_k\leq b<\infty$ and $V_k\equiv 0$ on $\Omega^c$. If
 Assume that 
 $$ 
  \int_{\Omega}V_k e^{2mu_k(y)}dy= \Lambda<\frac{\Lambda_1}{2}.
 $$
% Then there exists $x_0\in S_\vp^*$ such that $$\liminf _{k\to\infty}\int_{B_\ve(x_0)}V_kdx>0,\quad\text{for every }\ve>0,$$ 
 Then $c_k\to\infty$, $c_k=o(k)$ and  
 $$I_k(x):=\frac{1}{\gamma_{2m}}\int_{\Omega}\log\left(\frac{1+|y|}{|x-y|}\right)V_k e^{2mu_k(y)}dy,\quad x\in\R^{2m},$$
 is locally uniformly bounded from above on $\Omega\setminus S_\vp$, and locally uniformly bounded from below on $\R^{2m}$. In particular,
 $u_k\to\infty$ on $S_\vp^*$ and $u_k\to-\infty$ locally uniformly on $\Omega\setminus S_\vp$.
  \end{lem}
\begin{proof}
 For any fixed $R>0$ and $x\in B_R$ we bound
\[
\begin{split}
I_k(x)&=\int_{|y|\leq 2R,\,y\in\Omega}\log\left(\frac{1+|y|}{|x-y|}\right)V_k e^{2mu_k(y)}dy  +\int_{|y|>2R,\,y\in\Omega}\log\left(\frac{1+|y|}{|x-y|}\right)V_k e^{2mu_k(y)}dy\\
&\geq -C(R)+\int_{|y|>2R,\,y\in\Omega}\log\left(\frac 12+\frac{1}{2|y|}\right)V_k e^{2mu_k(y)}dy\\
&\geq -C(R).
\end{split}
\]
Since $\Lambda<\frac{\Lambda_1}{2}$, using Jensens inequality we obtain for some $p<2m$
\begin{align*}
 e^{2mu_k(x)}\leq e^{2mc_k}e^{-2m\beta|x|^2+2mk\vp(x)}\int_{\R^{2m}}\left(\frac{1+|y|}{|x-y|}\right)^pV_k(y)e^{2mu_k(y)}dy.
\end{align*}
Using that 
\begin{align*}
 \int_{\Omega}\left(\frac{1+|y|}{|x-y|}\right)^pe^{-2m\beta|x|^2+2mk\vp(x)}dx
 %&\leq \int_{\Omega}\left(\frac{1+|x-y|+|x|}{|x-y|}\right)^pe^{-2m\beta|x|^2+2mk\vp(x)}dx\\
 \xrightarrow{k\to\infty}0,
\end{align*}
and together with Fubini theorem, one has
\begin{align*}
 \int_{\Omega}V_k(x)e^{2mu_k(x)}dx
 =e^{2mc_k}o(1),\quad\text{as }k\to\infty. 
 %be^{2mc_k}\int_{\R^{2m}}V_k(y)e^{2mu_k(y)}dy\int_{\Omega}\left(\frac{1+|y|}{|x-y|}\right)^pe^{-2m\beta|x|^2+2mk\vp(x)}dx
\end{align*}
Now $\Lambda>0$ implies that $c_k\to\infty$. 

We assume by contradiction that $c_k\neq o(k)$. Then for some $\ve>0$ we have $\frac{c_k}{k}\geq 2\ve$ for $k$ large. 
Let $x_0\in S_\vp^*$ be such that \eqref{condVk} holds. Let $\delta>0$ be such that
$\vp(x)>-\ve$ for $x\in B_\delta(x_0)\cap\Omega$. Therefore 
$$u_k(x)\geq -C-k\ve+c_k\geq -C+k\ve,\quad x\in B_\delta(x_0)\cap\Omega,$$ and hence 
$$\int_{\Omega}V_ke^{2mu_k}dx\geq e^{-C+k\ve}\int_{B_\delta(x_0)}V_kdx\xrightarrow{k\to\infty}\infty,$$ a contradiction.

Now we prove that  $I_k$ is locally uniformly bounded from above on $\Omega\setminus S_\vp$. For $\tilde\Omega\Subset \Omega\setminus S_\vp$ we have 
$$k\vp+c_k\to-\infty\quad\text{uniformly on } \tilde\Omega.$$ Using Jensens inequality one can show that $\|e^{2mu_k}\|_{L^p(\Omega_1)}\leq C$ 
for some $p>1$, where $\tilde\Omega\Subset\Omega_1\Subset\Omega\setminus S_\vp$. For $x\in\tilde\Omega$ we obtain by   H\"older inequality 
\begin{align*}
 I_k(x)&= \frac{1}{\gamma_{2m}}\int_{\Omega_1^c\cap\Omega}\log\left(\frac{1+|y|}{|x-y|}\right)V_k e^{2mu_k(y)}dy +
                 \frac{1}{\gamma_{2m}}\int_{\Omega_1\cap\Omega}\log\left(\frac{1+|y|}{|x-y|}\right)V_k e^{2mu_k(y)}dy\\
  &\leq C+ C\|\log|x-\cdotp|\|_{L^{p'}(\Omega_1)}  \|e^{2mu_k}\|_{L^p(\Omega_1)}    \\
  &\leq C.
\end{align*}
The remaining part of the lemma follows immediately. 
\end{proof}

\noindent\emph{Proof of Corollary \ref{corollary}.}
Let $g\in C^\infty(\de\Omega)$ be such that $g\leq 0$, $g\not\equiv 0$ on $\de\Omega$ and $g=0$ on $\Gamma$. Let $\varphi$ be the solution to
\begin{equation}\notag
\begin{cases}
(-\Delta)^m\varphi=0\quad&\text{ in $\Omega$,}\\
(-\Delta)^j\varphi= 0\quad &\text{ on $\de\Omega$, \quad $j=1,\dots,m-1$}\\
\varphi=g\quad&\text{ on }\de\Omega.
\end{cases}
\end{equation}
Then by maximum principle $\varphi<0$ in $\Omega$ and hence $S^*_\varphi=\Gamma$. Then the conclusion follows by Theorem \ref{thm-1} and \ref{thm-2}.
\hfill$\square$

\medskip

%We conclude this section giving an example of a sequence of solutions corresponding to Remark \ref{remthm2}, showing that assumption \eqref{condVk} is a necessary condition.
\begin{prop}\label{non}
 Let $\Omega$ be a  domain in $\R^{2m}$. Let $\vp\in \mathcal{K}(\Omega,\emptyset)$. 
 Let $\tilde\Omega\Subset \Omega\setminus S_\vp$ be an open set.  Let $V_k$ be such that $V_k\equiv 0$ on $\tilde{\Omega}^c$ and $V_k\equiv 1$ on $\tilde \Omega$. Then for any $\Lambda>0$ there exists 
 no sequence $(u_k)$ of solutions to \eqref{eq1} satisfying \eqref{ukinfty} and \eqref{eq-2}.
\end{prop}
\begin{proof}
 We assume by contradiction that the statement of the proposition is not true. Then there exists a sequence of solutions $(u_k)$ to 
 \eqref{eq1} satisfying \eqref{ukinfty} and \eqref{eq-2} for some $\Lambda>0$. Therefore, by \eqref{ukinfty}, $u_k\to-\infty$ uniformly
 in $\tilde\Omega$ and hence
 \begin{align*}
  \Lambda=\int_{\Omega}V_ke^{2mu_k}dx=\int_{\tilde\Omega}e^{2mu_k}dx\xrightarrow{k\to\infty}0,
 \end{align*}
a contradiction.
\end{proof}

%Let $\Omega=B_3\subset\R^{2m}$ with $m\geq 2$, $0<\Lambda<\frac{\Lambda_1}{2}$. Let $V_k=1$ on $B_2\setminus B_1$ and $V_k=0$ outside. Set  $\vp(x)=-|x|^2$.
%Then arguing as in the proof of Theorem \ref{thm-2} we can find a sequence $u_k$ given by
%$$u_k(x)=\frac{1}{\gamma_{2m}}\int_{\Omega}\log\left(\frac{1}{|x-y|}\right)V_k e^{2mu_k(y)}dy+k\vp(x)+c_k$$ 
%and such that 
%\begin{align}
 %\int_{\Omega}V_k e^{2mu_k(y)}dy=\Lambda.\label{vol}
%\end{align}

%Then  we have $c_k\to\infty$. If $c_k=o(k)$ then $k\vp(x)+c_k\to-\infty$ on $B_\frac12^c$, which would contradict to \eqref{vol}. Therefore, 
%$c_k/k\geq 2\ve$ for some $\ve>0$, and hence, $u_k\to\infty$ on $B_\ve$. 

%%%%%%%%%%%%%%%%%%%%%%%%%%%%%%%%%%%%%%%%%%%%%%%%%%%%%%%%%%%%%%%%%%%%%%%%%%%%%%%%%%%%%%%%%%%%%%%%%%%%%%%%%%%%%%%%%%

%\begin{prop}[\cite{ARS,Mar1}]
%Let $\Omega$ be a regular domain in $\R^{2m}$. Let $u_k$ be a sequence of solutions to \eqref{eq1} with 
%$$\int_{\Omega}V_ke^{2mu_k}dx\leq C, $$ where $0\leq V_k\leq b<\infty$ is a sequence of functions on $\Omega$. If the sequence $(u_k)$ is not 
%pre-compact then there exists a sequence of real numbers  $\beta_k\to\infty$ and a function $\vp\in \mathcal{K}(\Omega,S_1)$ for some finite set $S_1$ such that 
%$$\frac{u_k}{\beta_k}\to \vp\quad \text{in }C^{2m-1}_{loc}(\Omega\setminus(S_1\cup S_\vp)).$$
%\end{prop}

\section{Proof of Theorem \ref{thm-4}}

\subsection{The case $\Omega$ is an annulus.}
Let $\Omega= B_{R_2}\setminus B_{R_1}$ be an annulus. Let $X=C^0_{rad}(\bar \Omega)$.  We fix  $\Lambda\in (0,\infty)$.
  For $k\in\N$ and $v\in X$ we choose $c_v=c(v,k)\in\R$ so that
$$\int_{\Omega}V_k e^{2m(v+c_v)}dx=\Lambda.$$ 
   Let $\vp\in \mathcal{K}(\Omega,\emptyset)$ be radially symmetric. % satisfies \eqref{Condphi}. 
For $k\in\N$ we define an operator 
$T_k:X\to X$, $v\mapsto\bar{v}$ where 
$$\bar{v}:=\tilde{v}+k \vp(x),\quad \tilde{v}(x)=\int_{\Omega}G(x,y)V_k(y)e^{2m(v(y)+c_v)}dy,$$
and $G$ is the Green function of $(-\D)^m$ on $\Omega$  with the Navier boundary conditions.

\begin{lem}\label{cpt-omega}
 Let $k\in\N$ be fixed. Let $(v,t)\in X\times (0,1]$ satisfies $v=tT_k(v)$. Then there exists $M>0$ such that $\|v\|_X\leq M$ for all
 such $(v,t)$.
\end{lem}
\begin{proof}
 We have 
 $$v(x)=t\int_{\Omega}G(x,y)V_k(y)e^{2m(v(y)+c_v)}dy+tk\vp(x)\geq -C(k)\quad \text{in } \Omega. $$
 Hence from the definition of $c_v$ we get  
 \[
 \Lambda=\int_\Omega V_k e^{2m(v+c_v)}\, dx\geq e^{2m(-C(k)+ c_v)} \int_\Omega V_k\, dx>a  e^{2m(-C(k)+ c_v)}
 \]
hence  $c_v\leq C(k)$. %If  $\Lambda<\frac{\Lambda_1}{2}$, by Jensen's inequality, for some $p>1$ we have 
 %$$\int_{\Omega}e^{2mp(v+c_v)}dx\leq C\int_{\Omega}e^{2mpv}dx\leq C.$$
%In the case $\Lambda\geq \frac{\Lambda_1}{2}$,  $\Omega=B_2\setminus B_\frac12$ and $v$ radially symmetric, 
 Define the cone $\C$ as the set
 \begin{equation}\label{C} 
\C:=\left\{x\in\Omega\colon |\bar x|\leq\rho x_1\right\}, \quad\text{ with } x=(x_1,\bar x)\in\R\times\R^{2m-1},
\end{equation} 
for some  $\rho>0$ to be fixed later.
For some finite $M=M(\rho)$ we can write $\Omega$ as a union of (not necessarily disjoint) cones $\{\C_i\}_{i=1}^M$ such that for each such cone $\C_i$ we have 
\begin{itemize}
\item[$(i)$]$\C_i$ is congruent to $\C$,
\item[$(ii)$]  $\int_{N(\C_i)}V_k(y)e^{2m(v(y)+c_v)}dy\leq \frac{\Lambda_1}{4},\quad N(\C_i):=\cup_{ \C_i\cap\C_j\neq \emptyset}\C_j$
\end{itemize}
and we fix $\rho$ such that $(ii)$ holds.
 %$$\int_{N(\C_i)}V_k(y)e^{2m(v(y)+c_v)}dy\leq \frac{\Lambda_1}{4},\quad N(\C_i):=\cup_{\bar C_i\cap\bar\C_j\neq \emptyset}\C_j.$$
 Notice that there exists $\delta>0$ such that $dist(\C_i,N(\C_i)^c)\geq\delta$ for  $i=1,\dots,M$. Therefore, for $x\in \C_1$ 
 $$v(x)\leq t\int_{N(\C_1)}G(x,y)V_k(y)e^{2m(v(y)+c_v)}dy+tk\vp(x)+C(\delta),$$ and together with Jensen's inequality, 
 for some $p>1$ we get 
 \begin{align*}
   \int_{\Omega}e^{p2m(v+c_v)}dx \leq M\int_{\C_1}e^{p2m(v+c_v)}dx\leq C.
 \end{align*}
% Now we can proceed as in the case $\Lambda<\frac{\Lambda_1}{2}$.
 Since $\varphi$ is radially symmetric and polyharmonic we have $\varphi\in C^{2m}(\bar\Omega)$, and therefore by elliptic estimates and Sobolev embeddings
 $$\|v-tk\varphi\|_X\leq C\|v-tk\varphi\|_{W^{2m,p}(\Omega)}\leq C\|(-\Delta)^m v\|_{L^p(\Omega)}\leq C,$$
concluding the proof.
 \end{proof}

A consequence of Lemma \ref{cpt-omega} is that for every $k\in\N$, the operator $T_k$ has a fixed point $v_k\in X$. 
We set $u_k=v_k+c_{v_k}$. Then 
\begin{align}\label{u-ell}
 u_k(x)=\int_{\Omega}G(x,y)V_k e^{2mu_k(y)}dy+k \vp(x)+c_{v_k},\quad \int_{\Omega}V_k e^{2mu_k(y)}dx=\Lambda.
\end{align}
We claim that $c_{v_k}\to \infty$.

Again writing $\Omega$ as a union of cones and using Jensen's inequality we obtain  
%If $\Lambda<\Lambda_1/2$, with the help of Jensen inequality, from \eqref{u-ell} one can get
$$\int_{\Omega}e^{2mu_k}dx\leq Ce^{2mc_{v_k}}\int_{\Omega}e^{2mu_k(y)}dy\int_{\Omega}\frac{e^{2mk\vp(x)}}{|x-y|^p}dx,$$ for some $p<2m$. Hence, if $c_{v_k}\leq C$, then
$$\int_{\Omega}V_ke^{2mu_k}dx \leq Cb\int_{\Omega}e^{2mu_k(y)}dy\int_{\Omega}\frac{e^{2mk\vp(x)}}{|x-y|^p}dx\xrightarrow{k\to\infty}0,$$
a contradiction. Thus $c_{v_k}\to \infty$, and hence $u_k\to\infty$ on $S^*_\vp$. 

%Therefore, $c_{v_k}\to\infty$, otherwise, as $\vp\leq 0$ (and the zero set of $\vp$  has dimension at most $2m-1$), the right hand side goes to
 %$0$ as $k\to\infty$, a contradiction to \eqref{u-ell}. 
 %If $\Lambda\geq\Lambda_1/2$ then writing $\Omega$ as a union of Cones (as in Lemma \ref{cpt-omega}) we arrive at the same conclusion.
 %We claim that $c_{v_k}=o(k)$ as $k\to\infty$.
 %If the claim is false then for some  $\delta>0$ we have $c_{v_k}\geq\delta_k$.  
 %We choose $x_0\in S_{\varphi}$ and $\ve>0$ such that $\vp\geq -\frac{\delta}{2}$ on $B_\ve(x_0)\cap\Omega$. 
 %Then $u_k\geq \frac12\delta_k$ on $B_\ve(x_0)\cap\Omega$ gives a contradiction to \eqref{u-ell}.
 %follows from the fact that $S_{\varphi}\neq\phi$.
 
 It remains to show that $u_k\to-\infty$ in $C^0_{loc}(\Omega\setminus S_{\varphi})$.
 Arguing as in Lemma \ref{lemmause} we conclude the proof. \hfill$\square$
 
% there exists $\beta_k\to\infty$ such that 
% $$\frac{u_k}{\beta_k}\to \tilde{\vp}\quad\text{in } C^0_{loc}(\Omega\setminus\tilde{S}_0),$$
% where 
% $$\D^m\tilde{\vp}=0\quad\text{in }\Omega,\quad \tilde{\vp}\leq 0\text{ in }\Omega,\,\tilde{\vp}\not\equiv0, \,\tilde{S}_0:=\{x\in\Omega:\tilde{\vp}(x)=0\}.$$
% %It is easy to see that $S_{\varphi}\subset\tilde{S}_0$. 
% We fix $x_0\in\Omega\setminus\tilde{S}_0$ and $\ve>0$ such that  $B_{2\ve}(x_0)\cap\tilde{S}_0=\phi$ and $u_k\leq C$ on $B_{2\ve}(x_0)$. Then 
% $$u_k(x)=k(\vp(x)+o(1))+ O(1),\quad\text{for }x\in B_\ve(x_0),$$ and hence 
% $$\frac{u_k}{\beta_k}=(\vp+o(1)) \frac{k}{\beta_k}\to\tilde{\vp}<0\quad\text{in }C^0(B_\ve(x_0)).$$
% This shows that $\frac{k}{\beta_k}\to c_0>0$ and in particular, $\tilde{\vp}=c_0\vp$ in $B_\ve(x_0)$. 
% In fact, in a similar way one can show  that $\tilde{\vp}=c_0\vp$ in $\Omega$ (one can also use the fact that if
% $\D^m w=0$ in $\Omega$ and $w=0$ on an open set $U\subset\Omega$ then $w=0$ in $\Omega$). 
% This finishes the proof of Theorem \ref{thm-4} when the domain  $\Omega$ is an annulus.
% 
 
\subsection{The case $\Omega$ is a ball }
We consider $$X=C^2_{rad}(\bar B_R),\quad \|v\|_X:=\max_{\bar B_R}(|v(x)|+|v'(x)|+|v''(x)|).$$ 
Let $\Lambda>0$. We fix $k\in \N$.  For $v\in X$ define $c_v\in\R$ given by 
$$\int_{\Omega}V_ke^{2m(v+c_v)}dx=\Lambda.$$
We define $T_k:X\to X$ given by  $v\mapsto \bar v$ where 
$$\bar v(x)=\frac{1}{\gamma_{2m}}\int_{\Omega}\log\left(\frac{1}{|x-y|}\right)V_k(y)e^{2m(v(y)+c_v)}dy+\left(k+\frac{|\D v(0)|}{2\D \vp(0)}\right)\vp(x).$$
Arguing as in \cite{H-vol} one can show that the operator $T_k$ has a fixed point, say $v_k$. We set $u_k=v_k+c_{v_k}$. Then 
$$u_k(x)=\frac{1}{\gamma_{2m}}\int_{\Omega}\log\left(\frac{1}{|x-y|}\right)V_k(y)e^{2mu_k(y)}dy+\left(k+\frac{|\D v_k(0)|}{2\D \vp(0)}\right)\vp(x)+c_{v_k},$$
and $$\int_{\Omega}V_ke^{2mu_k}dx=\Lambda.$$
Again as in \cite{H-vol} one can show that  there exists $C>0$ such that $u_k\leq C$ on $B_\ve$ for some $\ve>0$. 
Using this, and   as in the annulus domain case, one can show that
$ c_{v_k}\to\infty $. % \quad\text{and } c_{v_k}=o\left(k+\frac{|\D v_k(0)|}{2\D \vp(0)}\right). $$
Thus $u_k(x)\to\infty$ for every $x\in S^*_\vp$. Finally, similar to  the annulus domain case, it follows that $u_k\to -\infty$ locally uniformly in $\Omega\setminus S_\vp$.
\hfill$\square$ 

\section{Proof of Theorem \ref{thm-5}}

Let $m\geq 2$. We set  $$\vp_k(r,\theta):=r^k\cos(k\theta),\quad 0\leq r\leq 2,\,0\leq\theta\leq 2\pi.$$ 
We extend $\vp_k$ on $B_2 \subset \R^{2m}$ as a function of only two variables, that is, $\vp_k(x):=\vp_k(r,\theta)$ for $x\in B_2$,
where $(r,\theta)$ is the polar coordinate of $\Pi(x)$ and $\Pi:\R^{2m}\to\R^2$ is the projection map. Then $\vp_k$ is a harmonic function on
$B_2$. Let $\Phi_k$ be the solution to the equation 
$$
\left\{\begin{split} -\D\Phi_k=\vp_k &\quad\text{in }B_2,\\
 \Phi_k=0 &\quad \text{on }\partial B_2.
 \end{split}\right.
 $$
We fix $0<\Lambda<\Lambda_1$. Then by Proposition \ref{sol-prop} there exists a sequence of solutions $(w_k)$ to \eqref{w-ell} with 
$$f:=\frac{2\Lambda}{\Lambda_1}(-\D)^m u_0,\quad K_k:=\left\{ \begin{array}{ll}
                                                e^{2m(\Phi_k+\frac{2\Lambda}{\Lambda_1}u_0)}\quad&\text{on }B_2\\
                                                0\quad&\text{on }B_2^c,
                                                            \end{array}\right.
$$ where $u_0\in C^\infty(\R^{2m})$ with $u_0=-\log|x|$ on $B_1^c$. Then $$u_k:=w_k+c_{w_k}+\Phi_k+\frac{2\Lambda}{\Lambda_1}u_0$$ satisfies 
\eqref{Qcurv2} and $u_k$ is given by
%$$w_k(x)=\frac{1}{\gamma_{2m}}\int_{B_2}\log\left(\frac{1+|y|}{|x-y|}\right)K_k(y)e^{2m(w_k(y)+c_{w_k})}dy-\frac{2\Lambda}{\Lambda_1}u_0+a_k,$$ 
%for some $a_k\in\R$.
%Then as in the proof of Theorem \ref{thm-2}, for each $k\in\N$ there exists a solution $u_k$ to 
$$u_k(x)=\frac{1}{\gamma_{2m}}\int_{B_2}\log\left(\frac{1+|y|}{|x-y|}\right)e^{2mu_k(y)}dy+\Phi_k(x)+c_k,$$
for some $c_k\in\R$.
%where $c_k$ is chosen to satisfy \eqref{Qcurv2}.
%Then $u_k$ satisfies \eqref{eq1} with $V_k\equiv 1$, $\Omega=B_2$ and
Moreover, 
$$\Delta u_k=-\varphi_k +e_k,$$
where
$$|e_k(x)|\le C\int_{B_2}\frac{e^{2mu_k(y)}}{|x-y|^2}dy.$$ 
Integrating, using Fubini's theorem and \eqref{Qcurv2} we obtain $\|e_k\|_{L^1(B_2)}\le C$. Then \eqref{Deltauk} follows at once from the definition of $\varphi_k$. \hfill$\square$

\end{document}